\documentclass[11pt]{amsart}

\usepackage{amsfonts}
\usepackage{amssymb,stmaryrd, latexsym, amsthm, amsmath, verbatim}
\usepackage{indentfirst}

\allowdisplaybreaks

\numberwithin{equation}{section}
\theoremstyle{definition}
\newtheorem{Thm}[equation]{Theorem}
\newtheorem{Prop}[equation]{Proposition}
\newtheorem{Cor}[equation]{Corollary}
\newtheorem{Lem}[equation]{Lemma}

\newtheorem{Exa}[equation]{Example}
\newtheorem{Rmk}[equation]{Remark}

\makeatletter
\def\imod#1{\allowbreak\mkern5mu{\operator@font mod}\,\,#1}
\makeatother

\begin{document}
\title [Weakly Holomorphic Modular Forms
and hyperbolic Kac-Moody algebras]
{Weakly Holomorphic Modular Forms and \\ Rank two hyperbolic Kac-Moody algebras}
\author[H. H. Kim]{Henry H. Kim$^{\star}$}
\thanks{$^{\star}$ partially supported by an NSERC grant.}
\address{Department of
Mathematics, University of Toronto, Toronto, ON M5S 2E4, CANADA and Korea Institute for Advanced Study, Seoul, Korea}
\email{henrykim@math.toronto.edu}
\author[K.-H. Lee]{Kyu-Hwan Lee}
\address{Department of
Mathematics, University of Connecticut, Storrs, CT 06269, U.S.A.}
\email{khlee@math.uconn.edu}
\author[Y. Zhang]{Yichao Zhang}
\address{Department of
Mathematics, University of Connecticut, Storrs, CT 06269, U.S.A.}
\email{yichao.zhang@uconn.edu}

\subjclass[2010]{Primary 11F22; Secondary 17B67, 11F41}
\begin{abstract}
In this paper, we compute basis elements of certain spaces of weight $0$ weakly holomorphic modular forms and consider the integrality of Fourier coefficients of the modular forms. We use the results to construct automorphic correction of the rank $2$ hyperbolic Kac-Moody algebras $\mathcal H(a)$, $a=4,5,6$, through Hilbert modular forms explicitly given by Borcherds lifts of the weakly holomorphic modular forms. We also compute asymptotic of the Fourier coefficients as they are related to root multiplicities of the rank $2$ hyperbolic Kac-Moody algebras. This work is a continuation of the paper \cite{kim2012rank}, where automorphic correction was constructed for $\mathcal H(a)$, $a=3, 11, 66$. 
\end{abstract}

\maketitle

\section*{Introduction}

The relationship between affine Kac-Moody algebras and Jacobi modular forms is well understood through the works of Macdonald \cite{macdonald1972affine}, Kac, Peterson, Wakimoto \cite{kac1984infinite,kac1988modular} and others. More mysterious relationship between hyperbolic Kac-Moody algebras and automorphic forms was perceived by Lepowsky and Moody \cite{lepowsky1979hyperbolic} and Feingold and Frenkel \cite{feingold1983hyperbolic}, and further investigated by Borcherds \cite{borcherds1992monstrous}, Gritsenko and Nikulin \cite{gritsenko1997siegel}. Surprisingly, it turned out that hyperbolic Kac-Moody algebras should be ``corrected" to be precisely related to automorphic forms. Namely, hyperbolic Kac-Moody algebras need to be extended to generalized Kac-Moody superalgebras so that the denominator functions may become automorphic forms. The resulting generalized Kac-Moody superalgebras and automorphic forms are  called {\em automorphic correction} of the hyperbolic Kac-Moody algebras. 

In a series of papers \cite{gritsenko1996igusa,gritsenko1998automorphic-1,gritsenko1998automorphic-2,gritsenko1997siegel}, Gritsenko and Nikulin constructed automorphic correction of many rank $3$ hyperbolic Kac-Moody algebras. In a recent paper of 
Kim and Lee \cite{kim2012rank}, it was shown that rank $2$ symmetric hyperbolic Kac-Moody algebras $\mathcal H(a)$, $a \ge 3$, form infinite families through chains of embeddings. (See \cite{kang1995rank, kim2012rank} for the definition of $\mathcal H(a)$.)
Moreover, they considered three specific families and constructed automorphic correction for the first algebra in each family, i.e. $\mathcal H(3), \mathcal H(11)$ and $\mathcal H(66)$. Their construction of automorphic correction utilized Hilbert modular forms given by Borcherds products and written explicitly by Bruinier and Bundschuh \cite{bruinier2003borcherds}. 

Only three families could be considered in \cite{kim2012rank} because our knowledge on weakly holomorphic modular forms was limited  at that time. In order to generalize the results, we need to extend the one-to-one correspondence between vector-valued modular forms and scalar-valued modular forms and to write explicitly Borcherds product lifts of the scalar-valued modular forms.  
For the odd prime discriminant forms, this was done by Bruinier and Bundschuh \cite{bruinier2003borcherds}; we need to consider more general discriminant forms.

Furthermore, it is necessary to compute basis elements of the spaces of weight $0$ scalar-valued modular forms and to show integrality of their Fourier coefficients. Since such Fourier coefficients are to be interpreted as root multiplicities of generalized Kac-Moody superalgebras, the integrality is crucial. In the previous paper \cite{kim2012rank}, the integrality was shown only partially and was assumed to be true.

In this paper, we overcome the obstacles mentioned above and consider more general families of rank $2$ hyperbolic Kac-Moody algebras 
$\mathcal H(a)$ attached to the generalized Cartan matrix $\begin{pmatrix} 2 & -a \\ -a & 2 \end{pmatrix}$. Let
$a^2-4=Ns^2$, $s \in \mathbb N$, and suppose that $N$ is the fundamental discriminant of $\mathbb Q(\sqrt{a^2-4})$. Then we obtain automorphic correction of $\mathcal H(a)$ in the cases $N=12, a=4$; $N=8$, $a=6$; $N=21$, $a=5$.  
Actually, a generalization of the one-to-one correspondence between vector-valued forms and scalar-valued forms has already been established by Y. Zhang in his recent preprint \cite{zhang2013vector}, and we use the results throughout this paper.

After introducing notations in Section \ref{Notations}, we compute some basis elements $f_m$ $(m\ge 1)$ with the principal part 
$\frac 1{s(m)} q^{-m}$ in the spaces of weight $0$ scalar-valued modular forms for the discriminants $N=12, 8, 21$, where $s(m)$ is a normalizing factor. We start with $\eta$-quotients and use SAGE  to compute the Fourier coefficients of $f_m$ up to $q^{20}$. An outline of the computation and a presentation of the results are contents of Section \ref{hyperbolic}. The integrality of Fourier coefficients is considered in Section \ref{sec-integ}. We use Sturm's theorem \cite{sturm1987congruence} and show how one can check the integrality of Fourier coefficients. In particular, we check the integrality for $f_1$ in each of the cases $N=12, 8, 21$.

In Section \ref{sec-BorPro}, we explicitly write the Borcherds products associated to the scalar-valued modular forms 
following \cite{borcherds1998automorphic, zhang2013vector, bruinier2003borcherds, bruinier2008hilbert}. They are Hilbert modular forms on the quadratic extension $\Bbb Q[\sqrt{N}]$.
After that, we obtain automorphic correction of the hyperbolic Kac-Moody algebras $\mathcal H(a)$, $a=4,5,6$, 
in Section \ref{sec-AutCor}, using the construction in parallel with that of the cases considered in \cite{kim2012rank}. Namely, the Borcherds product associated to $f_1$ provides the automorphic correction (Theorem \ref{auto-corr}).
As mentioned above, Fourier coefficients of $f_m$ are related to root multiplicities of the rank $2$ hyperbolic Kac-Moody algebras, and it is valuable to know their asymptotic behavior. In Section \ref{sec-Asym}, we investigate asymptotics of the Fourier coefficients of 
some $f_m$'s using the method of Hardy-Ramanujan-Rademacher and show that, for $f_N$, $N=12, 8$, all the Fourier coefficients are non-negative. 

For other rank $2$ symmetric hyperbolic Kac-Moody algebras, the situation is somewhat different since the obstruction spaces of weight 
$2$ cusp forms are nontrivial (see Lemma \ref{obstruction}). In fact we prove that $f_1$ does not exist if $N>21$ (Lemma \ref{lem-own}). Hence for $N>21$, we need new ideas in order to construct automorphic correction. We hope that we can come back to these issues in the future.

\subsection*{Acknowledgments} We thank J. Bruinier who explained to us how to extend the result in his paper with 
Bundschuh \cite{bruinier2003borcherds} to general discriminant forms.

\vskip 1cm

\section{Notations}\label{Notations}

In this section, we recall some definitions and notations from \cite{zhang2013vector}.
Let $N_1>1$ be a square-free integer. Let $F=\mathbb Q(\sqrt{N_1})$ and $\mathcal O_F$ be its ring of integers. Let $N$ be the discriminant of $F/\mathbb Q$. Thus if $N_1\equiv 2,3\imod 4$, $N=4N_1$ and $\mathcal O_F=\mathbb Z[\sqrt{N_1}]$, and if $N_1\equiv 1\imod 4$, $N=N_1$ and $\mathcal O_F=\mathbb Z\left[\frac{\sqrt{N_1}+1}{2}\right]$.  Let $\mathrm{N}(x)$ and $\mathrm{tr}(x)$ denote the norm and trace of $x \in F/\mathbb Q$ respectively.
If $\mathfrak d$ is the different of $F/\mathbb Q$, we know that
\[\mathfrak d^{-1}=\{x\in F: \text{tr}(x\mathcal O_F)\subset \mathbb Z\}=
\left\{
\begin{array}{cl}
\frac{1}{2}\mathbb Z+\frac{1}{2\sqrt{N_1} }\mathbb Z, & \quad N_1\equiv 2,3\imod 4 ;\\
\frac{1}{\sqrt{N_1}} \mathcal O_F, &\quad N_1\equiv 1\imod 4.
\end{array}
\right.
\]
Define the following lattice $L=\mathbb Z^2\oplus \mathcal O_F$ with the quadratic form
\[q(a,b,\gamma)=\text{N}(\gamma)-ab,\quad a,b\in\mathbb Z, \gamma\in\mathcal O_F.\]
The corresponding bilinear form is given by
\[((a_1,b_1,\gamma_1),(a_2,b_2,\gamma_2))=\text{tr}(\gamma_1 \gamma'_2)-a_1b_2-a_2b_1.\] We see that $L$ is an even lattice of signature $(2,2)$. Its dual lattice is $L'=\mathbb Z^2\oplus \mathfrak d^{-1}$, hence the discriminant form is given by $D=L'/L\cong \mathfrak d^{-1}/\mathcal{O}_F$. The level of $D$ is $N$. Denote $q\imod 1$ on $D$ also by $q$.

Let $k$ be an even integer. Let $\rho_D$ be the Weil representation of $SL_2(\mathbb Z)$ on $\mathbb C[D]$; that is, if $\{e_\gamma:\gamma\in D\}$ is the standard basis for the group algebra $\mathbb C[D]$, then the action
\begin{align*}
\rho_D(T)e_\gamma &= e(q(\gamma))e_\gamma,\\
\rho_D(S)e_\gamma &=\frac{1}{\sqrt{N}}\sum_{\delta\in D}e(-(\gamma,\delta))e_\delta,
\end{align*}
defines the unitary representation $\rho_D$ of $SL_2(\mathbb Z)$ on $\mathbb C[D]$. Here $e(x)=e^{2\pi i x}$, and $T=\tiny{\begin{pmatrix} 1& 1\\0&1\end{pmatrix}}$, $S=\tiny{\begin{pmatrix}0&-1\\1&0 \end{pmatrix}}$ are the standard generators of $SL_2(\mathbb Z)$.

Let $\Gamma_0(N) \subseteq SL_2(\mathbb Z)$ be the congruence subgroup of matrices whose left lower entry is divisible by $N$. The weight $k$ slash operator on a function $f$ on the upper half plane is defined as
\[(f|_kM)(\tau)=(\text{det}M)^{\frac{k}{2}}(c\tau+d)^{-k}f(M\tau), \text{ for } M={ \scriptsize \begin{pmatrix}
a&b\\
c&d
\end{pmatrix} }\in GL_2^+(\mathbb R).\] We extend the definition of slash operator to a vector-valued form: for $F=\sum_\gamma F_\gamma e_\gamma$, we define $F|_k M= \sum_\gamma (F_\gamma|_k M) e_\gamma$. 
Now we let $\mathcal A(k,\rho_D)$ be the space of modular forms of weight $k$ and type $\rho_D$. That is, $F=\sum_\gamma F_\gamma e_\gamma\in\mathcal A(k,\rho_D)$ if $F|_kM=\rho_D(M)F$ for any $M\in SL_2(\mathbb Z)$ and $F_\gamma=\sum_{n\in q(\gamma)+\mathbb Z}a(\gamma,n)q^n$ with at most finitely many negative power terms. Let $\mathcal M(k,\rho_D)$ and $\mathcal S(k,\rho_D)$ denote the space of holomorphic forms and the space of cusp forms, respectively.
We denote by $\mathcal A^\text{inv}(k,\rho_D)$ the subspace of modular forms that are invariant under $\text{Aut}(D)$. Similarly, we have $\mathcal M^{\text{inv}}(k,\rho_D)$.

Given a Dirichlet character $\chi$ modulo $N$, we denote by $A(N,k,\chi)$ ($M(N,k,\chi)$, $S(N,k,\chi)$, respectively), the space of holomorphic functions $f$ on the upper half plane that satisfy
\[(f|_kM)(\tau)=\chi(d)f(\tau),\quad\text{for all } M={\scriptsize \begin{pmatrix}a&b\\c&d\end{pmatrix} } \in\Gamma_0(N),\] 
and that are meromorphic (or are holomorphic, or vanish, respectively) at all cusps. The functions  in the space $A(N,k,\chi)$ are called {\em weakly holomorphic}.

We define the character $\chi_D:=\left ( \frac N \cdot \right )$. For any positive integer $m$, we denote by $\omega(m)$ the number of distinct prime divisors of $m$. 
Define a subspace $A^\delta(N,k,\chi_D)$ of $A(N,k,\chi_D)$ for each $\delta=(\delta_p)_{p\mid N }\in\{\pm 1\}^{\omega(N)}$ as follows:
\[A^\delta(N,k,\chi_D)=\left\{f=\sum_n a(n)q^n\in A(N,k,\chi_D)\left|
 \begin{split}a(n)=0 \text{ if } (n,N)=1 \text{ and }\\ \chi_p(n)=-\delta_p \text{ for some }p\mid N\end{split}\right.\right\}.\]
The condition we impose on the Fourier coefficients of the functions in $A^\delta(N,k,\chi_D)$ will be called the \emph{$\delta$-condition}. Then by \cite{zhang2013vector}, Proposition 4.9, 
$A^\delta(N,k,\chi_D)=\bigoplus_{\delta} A^\delta(N,k,\chi_D)$, where $\delta$ runs over $\{\pm 1\}^{\omega(N)}$,

If $p$ is odd, we see that the $p$-component of $\chi$ is $\chi_p=\left(\frac \cdot p\right)$. If $2\mid N$, the $2$-component of $\chi$ is given by 
$$\chi_2=\begin{cases} \left(\frac{-4}{\cdot}\right), &\text{if $N_1\equiv 3\imod 4$}\\ 
\left(\frac{2}{\cdot}\right), &\text{if $N_1\equiv 2\imod 8$}\\ 
\left(\frac{-2}{\cdot}\right), &\text{if $N_1\equiv 6\imod 8$}.
\end{cases}
$$
We set $\chi_1$ to be the trivial character. For each positive integer $m$, write $\chi_m=\prod_{p\mid m}\chi_p$.
For each prime $p|N$, we define $\epsilon_p=\chi_p(-1)$ if $p$ is odd,
$\epsilon_2=-1$ if $p=2$ and $N_1\equiv 3\imod 4$, and
$\epsilon_2=\chi_{N_1/2}(-1)$ if $p=2$ and $N_1\equiv 2\imod 4$. Then we set $\epsilon=(\epsilon_p) \in\{\pm 1\}^{\omega(N)}$. 
We also define $\epsilon^*=(\epsilon_p^*)$ to be $\epsilon_p^*=\chi_p(-1)\epsilon_p$. Then we can see easily that
$\epsilon_p^*=1$ for each prime $p\mid N$. 
Note that if $p|N$, $p\equiv 3$ mod 4, then the Fourier expansion of $f\in A^\epsilon(N,k,\chi_D)$ is of the form 
$f=\sum_{n\geq n_0} a_n q^n$, $a_1=0$. We also note that if $f=q+O(q^2)\in A^\delta(N,k,\chi_D)$ for some $\delta$,
then $\delta=\epsilon^*$.

We shall employ the notation $p^l||N$ for a prime number $p$ and non-negative integers $l,N$, to mean that $p^l\mid N$ but $p^{l+1}\nmid N$.

\section{Computation of Basis Elements} \label{hyperbolic}

Recall the following lemma on the obstruction of the existence of weakly holomorphic modular forms. Let $s(m)=2^{\omega((m,N))}$ for each $m \in \mathbb Z$. 
\begin{Lem} \cite[Theorem 5.5]{zhang2013vector}\label{obstruction}
Let $P(q^{-1})=\sum_{n<0}a(n)q^n$ be a polynomial in $q^{-1}$ that satisfies the $\epsilon$-condition; namely $a(n)=0$ if $\chi_p(n)=-\epsilon_p$ for some $p\mid N$. There exists $f\in A^\epsilon(N,k,\chi_D)$ with prescribed principal part $P(q^{-1})$ if and only if 
$$\sum_{n<0} s(n)a(n)b(-n)=0,
$$
for each $g=\sum_{n>0} b(n)q^n\in S^{\epsilon^*}(N,2-k,\chi_D)$. 
\end{Lem}

We consider the cases of $N=12$, $N=8$ and $N=21$. With SAGE, we can easily see that in all of these cases we have $S(N,2,\chi_D)=\{0\}$, hence there are no obstructions for the existence of elements in
$A^\epsilon(N,0,\chi_D)$. We explicitly compute the basis elements of the space $A^\epsilon(N,0,\chi_D)$.  
We denote by $f_m$, $m\in\mathbb Z_{>0}$, the function in the space $A^\epsilon(N,0,\chi_D)$ with the principal part $\frac 1{s(m)} q^{-m}$, i.e. $f_m=\frac{1}{s(m)}q^{-m}+O(1)$.
We will compute the Fourier coefficients of $f_m$ up to $q^{20}$. We briefly explain how the computations are performed. Firstly, we compute the weight, character, and behavior at all cusps of the $\eta$-quotients of corresponding level (\cite{martin1996multiplicative}). Secondly, we compute the order of zeros of $f_m$ at all cusps. Then we multiply $f_m$ by a suitable quotient of those $\eta$-quotients to bring $f_m$ to a holomorphic modular form of certain weight and character. Finally, with SAGE, we can obtain a basis of modular forms of the same weight and character, and by solving a linear system we find our $f_m$.

Before we list our $f_m$ explicitly in  each case, we first prove the uniqueness of $f_m$ if it exists. Note that in the case of $2\nmid N_1$ this is known by Corollary 4.13 in \cite{zhang2013vector}. For the general case, we pass to the vector valued forms using the one-to-one correspondence in \cite{zhang2013vector}.

\begin{Lem}\label{Uniqueness-2mod4} 
If $f_m$ exists, then it is unique.
\end{Lem}
\begin{proof}
It is enough to prove that if $f\in A^\epsilon(N,0,\chi_D)$ is holomorphic at $\infty$ then $f=0$. The one-to-one correspondence from $A^\epsilon(N,0,\chi_D)$ to $\mathcal A^{\text{inv}}(0,\rho_D)$, constructed in \cite{zhang2013vector}, is denoted by $\psi$ and its inverse by $\phi$. Let $F=\psi(f)\in \mathcal A^{\text{inv}}(0,\rho_D)$. By Theorem 4.16 in \cite{zhang2013vector}, we see that $F_\gamma$ is holomorphic at $\infty$ for each $\gamma\in D$. Let $W=\text{span}_\mathbb{C}\{F_\gamma\}$ and $W'=\text{span}_\mathbb{C}\{F_0|M\colon M\in\text{SL}_2(\mathbb Z)\}$.
Therefore all functions in $W$, hence in $W'$, are holomorphic at $\infty$, since we know that $W=W'$ (\cite[Section 4]{zhang2013vector}). It follows that $F_0|M$ is holomorphic at $\infty$ for each $M\in SL_2(\mathbb Z)$ and $F_0\in M(N,0,\chi_D)$. So $F_0=0$, $F=0$ and $f=\phi(F)=0$. 
\end{proof}

\subsection{N=12}
In this case, $\epsilon_2=\epsilon_3=-1$.
By Corollary 4.13 and Theorem 5.5 in \cite{zhang2013vector}, we know that $f_m$ exists if and only if $m\equiv 0,1,4,6,9,10\imod 12$. We list the first few Fourier coefficients of $f_m$ as follows:
\begin{eqnarray*}
f_1&=&q^{-1}+1+2q^2+q^3-2q^6-2q^8+4q^{12}+4q^{14}-q^{15}-6q^{18}+O(q^{20}), \\
f_4&=&\frac{1}{2}q^{-4} + \frac{5}{2} - 2q^2 + 16q^3 + 22q^6 - 35q^8 - 160q^{11} +\frac{209}{2}q^{12}- 172q^{14}\\ &&\hspace{3cm}+ 416q^{15} + 390q^{18} + O(q^{20}), \\
f_6&=&\frac{1}{4}q^{-6} + 3 + \frac{27}{2}q^2 - 16q^3 + 36q^6 + 162q^8 - 864q^{11} +292q^{12} + 1080q^{14} \\ &&\hspace{3cm}- 1440q^{15} + 1629q^{18} + O(q^{20}), \\
f_9&=&\frac{1}{2}q^{-9} + 5 - 54q^2 + 6q^3 - 330q^6 + 1782q^8 + 54q^{11} + 4884q^{12}- 20844q^{14}\\ &&\hspace{3cm} - \frac{87}{2}q^{15} - 41822q^{18} + O(q^{20}), \\
f_{10}&=&\frac{1}{2}q^{-10} + 2 - 40q^2 - 160q^3 + \frac{1045}{2}q^6 - 1460q^8 + 11840q^{11} +9080q^{12}\\ &&\hspace{3cm} - 20235q^{14} - 59456q^{15} + 88440q^{18}+O(q^{20}), \\
f_{12}&=&\frac{1}{4}q^{-12} + \frac{3}{2} + 54q^2 + 144q^3 + 606q^6 + 3807q^8 + 35424q^{11} +
14184q^{12}\\ &&\hspace{3cm} + 69444q^{14} + 106144q^{15} + 177246q^{18} + O(q^{20}).
\end{eqnarray*}
For $m>12$, $f_m$ can be obtained by multiplying one of above by $j(12\tau)$ and then eliminating other negative power terms. 

\subsection{N=8}

Note that $\epsilon_2=1$. We know that $f_m$ exists if and only if $m\equiv 0,1,2,4,6,7\imod 8$. We can compute explicitly the following list:
\begin{eqnarray*}
f_1&=&q^{-1}+ 2 + 2q + 4q^2 - 4q^4 - 8q^6 + q^7 + 12q^8 - 2q^9 + 16q^{10}
- 24q^{12} - 32q^{14} \\ &&\hspace{1cm}- q^{15} + 44q^{16} + 4q^{17} + 60q^{18}+O(q^{20}), \\
f_2&=&\frac{1}{2}q^{-2} + 3 + 8q - 3q^2 + 14q^4 - 24q^6 - 64q^7 + 42q^8 + 120q^9- 80q^{10} + 132q^{12} \\ &&\hspace{1cm}-\frac{447}{2}q^{14} - 576q^{15} + 370q^{16} + 912q^{17} -573q^{18}+O(q^{20}),  \\
f_4&=&\frac{1}{2}q^{-4} + 5 - 16q + 28q^2 + 89q^4 + 280q^6 - 896q^7 + 730q^8 -
2288q^9 + 1744q^{10} \\ &&\hspace{1cm}+ 3984q^{12} + 8480q^{14} - 24448q^{15} + 17366q^{16} \\ 
&&\hspace{2cm} -48928q^{17} + 34212q^{18}+ O(q^{20}), \\ 
f_6&=&\frac{1}{2}q^{-6}+ 2 -48q - 72q^2 +420q^4 -1708q^6 + 6528q^7 + 6012q^8 - 21200q^9\\ &&\hspace{1cm} - \frac{36669}{2}q^{10} + 51128q^{12} - 133056q^{14} + 419200q^{15} +
325644q^{16}\\ &&\hspace{2cm} - 1000800q^{17} - 759864q^{18}  + O(q^{20})\\
f_7&=&q^{-7}+ 16 + 7q - 224q^2 - 1568q^4 + 7616q^6 + 128q^7 + 29792q^8 +
14q^9\\ &&\hspace{1cm} - 101248q^{10} - 310464q^{12} + 878336q^{14} - 896q^{15} +
2328928q^{16}\\ &&\hspace{2cm} - 7q^{17} - 5852448q^{18}  + O(q^{20}), \\
f_8&=&\frac{1}{2}q^{-8} + 9 + 96q + 168q^2 + 1460q^4 + 8016q^6 + 34048q^7 +
34737q^8 + 136608q^9\\ &&\hspace{1cm} + 130144q^{10} + 434472q^{12} + 1330368q^{14} +
4533504q^{15} + 3799986q^{16} \\ &&\hspace{2cm}+ 12556992q^{17} +10235352q^{18} + O(q^{20}). 
\end{eqnarray*}
As in the case of $N=12$, we can compute $f_m$ when $m>8$ using $j(8\tau)$ and the above list. 

\subsection{N=21}

This is a case when $N_1$ is composite. Note that in this case $\epsilon_3=\epsilon_7=-1$ and $\epsilon_3^*=\epsilon_7^*=1$. Therefore $f_m$ exists if and only if 
\[m\equiv 0,1,4,7,9,15,16,18\imod 21.\]
By similar computations in the previous cases, we obtain 
\begin{eqnarray*}
f_1&=&q^{-1} + \frac{1}{2} + q^3 + q^5 - q^6 - q^{14} - q^{17} + 2q^{20} + q^{21} + q^{24} -
2q^{27} - q^{33} - q^{35} - 2q^{38}\\
&&\hspace{2cm} + 3q^{41} + 2q^{42} + 3q^{45} + q^{47} - 4q^{48}
+ O(q^{49}).
\end{eqnarray*} 
Here we show more terms as the level is large. Similar computations can give all $f_m$ with $m\leq 21$ and then using $j(21\tau)$ we can have all $f_m$.

\medskip

We prove a lemma for later use in Section \ref{sec-AutCor}, which is interesting in its own right.
\begin{Lem} \label{lem-own}
The modular form $f_1$ exists if and only if $1< N \le 21$.
\end{Lem}

\begin{proof} The fundamental discriminants for $1<N \le 21$ are $N=5,8,12,13,17,21$. 
Now it can be checked that $S(N, 2, \chi_D)=\{ 0 \}$ for $1<N\le 21$. It follows from Lemma \ref{obstruction} that $f_1$ exists in $A^\epsilon(N,0,\chi_D)$. Now assume that $N > 21$. Using dimension formulas \cite[Section 6.3, page 98]{stein2007modular}, one can see that $S(N, 2, \chi_D) \neq \{ 0 \}$. Since $\chi_D$ is primitive, the space $S(N, 2, \chi_D)$ consists of newforms $f=\sum_{n \ge 1} b(n) q^n$. In particular, if $f$ is a Hecke eigenform, we have $b(1) \neq 0$. It follows that there exists $g=q+O(q^2)\in S^\delta(N,2,\chi_D)$ for some $\delta$. 
By the definition of $\epsilon^*$, $\delta=\epsilon^*$.
Hence by Lemma \ref{obstruction}, $f_1$ does not exist.
\end{proof}

\section{Integrality of Fourier Coefficients} \label{sec-integ}

In this section, we prove the integrality of Fourier coefficients of $f_1$ in above cases and the case when $N=17$, using Sturm's theorem. More precisely, if $f_1=\sum_na(n)q^n$, we prove that $s(n)a(n)\in\mathbb Z$. 
The case $N=17$ was considered in \cite{kim2012rank}, where the integrality was assumed to be true. 

For any congruence subgroup $\Gamma$, we denote by $M(\Gamma,k)$ the space of holomorphic modular forms of weight $k$ for $\Gamma$. For a commutative ring $R$, let $R\llbracket q\rrbracket$ be the ring of power series in $q$ over $R$. We begin with Sturm's theorem. 
\begin{Thm}[{\cite[Sturm]{sturm1987congruence}}] \label{Sturm}
Let $\Gamma$ be any congruence subgroup of $SL_2(\mathbb Z)$, $\mathcal O_F$ be the ring of integers in a number field $F$, and $\mathfrak p$ be any prime ideal. Assume $f=\sum_na_nq^n\in M(\Gamma, k)\cap \mathcal O_F\llbracket q\rrbracket $. If $a_n\in \mathfrak p$ for $n\leq \frac{k}{12}[SL_2(\mathbb Z):\Gamma]$, then $a_n\in\mathfrak p$ for all $n$.
\end{Thm}

\begin{Cor} \label{cor-Sturm}
Let $\Gamma$ be any congruence subgroup of $SL_2(\mathbb Z)$. Assume $f=\sum_na_nq^n\in M(\Gamma, k)\cap \mathbb Q\llbracket q\rrbracket $ with bounded denominator. If $a_n\in \mathbb Z$ for $n\leq \frac{k}{12}[SL_2(\mathbb Z):\Gamma]$, then $a_n\in\mathbb Z$ for all $n$.
\end{Cor}
\begin{proof}
Let $M$ be the smallest positive integer such that $Mf\in\mathbb Z\llbracket q\rrbracket $, and we need to prove that $M=1$. Suppose $M>1$ and let $p$ be any prime divisor of $M$.
Now $Mf\in M(\Gamma,k)\cap \mathbb Z\llbracket q \rrbracket$ and $p\mid Ma_n$ for all $n$ up to $\frac{k}{12}[SL_2(\mathbb Z):\Gamma]$. By Theorem \ref{Sturm}, we have $p\mid Ma_n$ for all $n$. Therefore, $p^{-1}Mf\in \mathbb Z\llbracket q\rrbracket $, contradicting the minimality of $M$.
\end{proof}

\begin{Exa}
In the case when $N=p=17$, we have $[SL_2(\mathbb Z):\Gamma_1(17)]=288$, so the Sturm's bound is $96$. By Proposition 8 in \cite{bruinier2003borcherds} or Proposition 5.7 in \cite{zhang2013vector}, $f_m$ has bounded denominator for each $m$. In this case we see that $\eta(\tau)^7\eta(17\tau)f_1\in M(\Gamma_1(17),4)$ (Mayer \cite{mayer2007hilbert} used a different $\eta$-product). Because of the constant term of $f_1$ is $1/2$, we consider the form $g=\eta(\tau)^7\eta(17\tau)(f_1-\frac{1}{2})$ instead. With SAGE, we can explicitly see that all of the first $96$ Fourier coefficients of $g$ are integral, hence all Fourier coefficients of $g$ are integral by Corollary \ref{cor-Sturm}. Therefore, all Fourier coefficients of $f_1-\frac{1}{2}$ are integral. We can do similarly the cases $N=12,8$ and $21$. Here we only briefly mention some details:
\begin{itemize}
\item $N=12$: we have $f_1(\tau)\eta(\tau)^2\eta(3\tau)^{-2}\eta(4\tau)\eta(6\tau)^2\eta(12\tau)\in M(\Gamma_1(12),2)$, so the Sturm's bound is $16$ and we readily see the integrality.
\item $N=8$: we have $f_1(\tau)\eta(\tau)^{-2}\eta(2\tau)^{3}\eta(4\tau)\eta(8\tau)^2\in M(\Gamma_1(8),2)$, so the Sturm's bound is $8$ and the integrality follows easily.
\item $N=21$: Because of the constant term $1/2$, we consider $\left(f_1(\tau)-\frac{1}{2}\right)\eta(\tau)^{12}\eta(3\tau)^{-3}\eta(7\tau)^3\in M(\Gamma_1(21),6)$. Then the Sturm's bound is $192$ and the integrality also follows via SAGE.
\end{itemize}
\end{Exa}

\begin{Rmk}
The integrality of $s(n)a(n)$ is expected to hold generally for a \emph{reduced modular form} like $f_m$. This type of integrality is more precise than the naive integrality $a(n)\in\mathbb Z$. This question is raised in \cite{zhang2013zagier}; see Section 6 therein for details.
\end{Rmk}

\section{Borcherds Products} \label{sec-BorPro}

In this section, we explicitly write Borcherds products corresponding to modular forms in $A^\epsilon(N,0,\chi_D)$. We will use the Hilbert modular forms given by these Borcherds products to establish automorphic correction of some rank $2$ hyperbolic Kac-Moody algebras in the next section. 

Let $F=\mathbb Q(\sqrt N_1)$ for $N_1>1$, a square-free integer, and let $N$ be the fundamental discriminant as before. We keep the notations in Section \ref{Notations}. We write $x'$ for the conjugate of an element $x\in F$. Then $\text{tr}(x)=x+x'$ and $N(x)=xx'$. Denote by $\varepsilon_0$ the fundamental unit in $F$; in particular $\varepsilon_0>1$.
Recall that we have the lattice $L=\mathbb Z^2\oplus \mathcal O_F$ with the quadratic form $q(a,b,\lambda)=N(\lambda)-ab$ for $a,b\in\mathbb Z$ and $\lambda\in\mathcal O_F$. We also have the dual lattice $L'=\mathbb Z^2\oplus \mathfrak d^{-1}$ of $L$, the discriminant form  $D=L'/L$, and $\chi_D=\left(\frac{N}{\cdot}\right)$.

Denote $\mathbf \Gamma_F=SL_2(\mathcal O_F)$ the Hilbert modular group. Let $\mathbb H$ be the upper half plane; we use $(z_1,z_2)$ as a standard variable on $\mathbb H\times \mathbb H$ and write $(y_1,y_2)$ for its imaginary part.
For every positive integer $m$, we denote by $T(m)$ the $\mathbf\Gamma_F$-invariant algebraic divisor on $\mathbb H\times \mathbb H$ defined by:
\[T(m)=\sum_{(a,b,\lambda)\in L'/\{\pm 1\}\atop ab-N(\lambda)=m/N}\left\{(z_1,z_2)\in\mathbb H\times\mathbb H: az_1z_2+\lambda z_1+\lambda' z_2+b=0\right\}.\] Moreover, define for $m>0$ the subset in $\mathbb R_{>0}\times \mathbb R_{>0}$:
\[S(m)=\bigcup_{\lambda \in \mathfrak d^{-1}\atop -N(\lambda)=m/N}\left\{(y_1,y_2)\in\mathbb R_{>0}\times \mathbb R_{>0}\colon \lambda y_1+\lambda' y_2=0\right\}.\] 

Fix a weakly holomorphic form $f=\sum_{n\in\mathbb Z}a(n)q^n\in A^\epsilon(N,0,\chi_D)$. Each connected component $\mathcal W$ of the space
\[\mathbb R_{>0}\times \mathbb R_{>0}-\bigcup_{n<0\atop a(n)\neq 0}S(-n)\]
is called a {\em Weyl chamber} associated to $f$, following the terminology in \cite{bruinier2008hilbert}. We say $\lambda\in F$ is positive with respect to a Weyl chamber $\mathcal W$, if $\lambda y_1+\lambda'y_2>0$ for all vectors $(y_1,y_2)\in \mathcal W$, in which case we write $(\lambda,\mathcal W)>0$. For each positive integer $m$ and each Weyl chamber $\mathcal W$, define
\[R(m,\mathcal W)=\{\lambda\in\mathfrak d^{-1}: \ N(\lambda)=-m/N, \ (\varepsilon_0^{-2}\lambda, \mathcal W)<0, \ (\lambda,\mathcal W)>0\}.\] 

When $N$ is a prime $p \equiv 1 \imod 4$, Bruinier and Bundschuh explicitly described Borcherds products corresponding to weight $0$ weakly holomorphic modular forms in \cite{bruinier2003borcherds}. With the above settings and computations, we extend their description to cover non-prime cases in Theorem \ref{thm-BB} below. The general construction of Borcherds can be found in his seminal paper \cite{borcherds1998automorphic}. 

\begin{Thm}[{Compare to \cite[Theorem 9]{bruinier2003borcherds}}] \label{thm-BB}
Let $f=\sum_{n\in\mathbb Z}a(n)q^n\in A^\epsilon(N,0,\chi_D)$ be such that $s(n)a(n)\in\mathbb Z$ for all $n<0$. Then there is a meromorphic function $\Psi(z_1,z_2)$ on $\mathbb H\times \mathbb H$ with the following properties:

(1) $\Psi$ is a meromorphic Hilbert modular form for $\mathbf\Gamma_F$ with some unitary character of finite order. The weight of $\Psi$ is equal to $s(0)a(0)/2$.

(2) The divisor of $\Psi$ is determined by the principal part (at $\infty$) of $f$ and equals \[\sum_{n<0}s(n)a(n)T(-n).\]

(3) Let $\mathcal W$ be a Weyl chamber attached to $f$. Then the function $\Psi$ has the Borcherds product expansion
\begin{equation} \label{eqn-BB} 
\Psi(z_1,z_2)=e(\rho z_1+\rho'z_2)\prod_{\nu\in\mathfrak d^{-1}\atop (\nu,\mathcal W)>0}\left(1-e(\nu z_1+\nu' z_2) \right)^{s(N\nu\nu')a(N\nu\nu')},
\end{equation}
where $e(z)=e^{2\pi iz}$. Furthermore,
the Weyl vector $\rho$ associated to $f$ and $\mathcal W$ is given by
\[\rho=\rho_{f,\mathcal W}=\frac{1}{\varepsilon_0^2-1}\sum_{m>0}s(-m)a(-m)\sum_{\lambda\in R(m,\mathcal W)}\lambda.\] The product converges normally for all $(z_1,z_2)$ with $y_1y_2\gg 0$.

(4) There exists a positive integer $c$ such that $\Psi^c$ has integral rational Fourier coefficients with greatest common divisor $1$.
\end{Thm}
\begin{proof}
Most of the statements follow from Borcherds's Theorem 13.3 in \cite{borcherds1998automorphic} and the isomorphism between vector-valued and scalar-valued modular form spaces (Theorem 4.16 in \cite{zhang2013vector}). For the computation of the Weyl vector, one can see Section 3.2 in \cite{bruinier2008hilbert}. The last statement follows from Proposition 5.7 in \cite{zhang2013vector}, except the case when $N_1\equiv 2\imod 4$. Actually, by a similar argument utilized in Lemma \ref{Uniqueness-2mod4}, we can see that Proposition 5.7 in \cite{zhang2013vector} is also true for the case $N_1\equiv 2\imod 4$ when $\delta$ is specified to be $\epsilon$ (or $\epsilon^*$). 
\end{proof}

\begin{Rmk}
When $f=q^{-1}+O(1)$, we can compute Weyl vectors more explicitly. For example, if we choose $\mathcal W$ to be the one that contains $(1,\varepsilon_0)$, then
\begin{equation} \label{Weyl-vector}
\rho_{f,\mathcal W}=\left\{
\begin{array}{cl}
\frac{\varepsilon_0}{\sqrt N}\frac{1}{\text{tr}(\varepsilon_0)} & \quad\text{if } N(\varepsilon_0)=-1,\\
\frac{1+\varepsilon_0}{\text{tr}(\sqrt N\varepsilon_0)} & \quad\text{if } N(\varepsilon_0)=1.\\
\end{array}
\right.
\end{equation} This formula is given in Example 3.11 in \cite{bruinier2008hilbert}.
\end{Rmk}

\subsection{Computing the Weights}

In this subsection, we explain  how to compute the weights of the  Hilbert modular forms in Theorem \ref{thm-BB} and will consider the case $N=12$ as an example.
Let $f_m$ be defined as above, and the corresponding vector-valued modular form will be denoted by $F_m=\sum_\gamma F_{\gamma,m}e_\gamma$. By Theorem 13.3 in \cite{borcherds1998automorphic}, the weight of $\Psi_f$ is given by $\frac{1}{2}a_0(0)$ where $a_0(0)$ is the constant coefficient of $F_{0,m}$. By Theorem 4.16 in \cite{zhang2013vector}, this is in turn given by
\[\frac{1}{2}a_0(0)=\frac{1}{2}s(0)a(0),\]
where $a(0)$ is the constant Fourier coefficient of $f_m$. According to this, the weights in the case of $N=12$, where $s(0)=4$, are given by
\[\begin{array}{c|c|c|c|c|c|c}
m& 1&4&6&9&10&12\\
\hline
a(0) &1& 5/2&3&5&2&3/2\\
\hline
\text{weight} &2& 5&6&10&4&3
\end{array}
\]

Alternatively, one can also compute the weight by the  theorem on obstructions (\cite[Theorem 5.5]{zhang2013vector}) as noted in \cite{bruinier2003borcherds}. More precisely, $a(0)$ is given by the Eisenstein series, $E^{\epsilon^*}=1+\sum_{n>0}B(n)q^n$, in the dual modular form space, as follows:
\[a(0)=-\frac{1}{s(0)}\sum_{n<0}s(n)a(n)B(-n),\]
and the weight is then $-\frac{1}{2}\sum_{n<0}s(n)a(n)B(-n)$. Therefore, the principal part determines the weight explicitly. It follows that if $s(n)a(n)\in\mathbb Z$ for all negative $n$, the weight is half integral (or integral). In the case $N=12$, the constant term $a(0)$ of $f_m$ is given by $-\frac{B(m)}{4}$. Since we have
\[E^{\epsilon^*}=1 - 4q - 10q^4 - 12q^6 - 20q^9 - 8q^{10} - 6q^{12} - 56q^{13} - 34q^{16}+O(q^{18}),\]
we  see that $a(0)$ matches the data given above for each $m$.

\begin{Rmk}
(1) The computation of $a(0)$ by $B(m)$ is a special case of a more general phenomenon, {\em Zagier duality}. See Theorem 5.7 in \cite{zhang2013zagier}; note that we normalize $E^{\epsilon^*}$ differently therein.

(2) Since $-\frac{B(m)}{2}$ represents the weight of some holomorphic Hilbert modular form, it makes sense that the coefficients $B(m) (m>0)$ are all negative and integral.
 \end{Rmk}

\section{Automorphic Correction} \label{sec-AutCor}

In this section, we construct automorphic correction of some rank $2$ hyperbolic Kac-Moody algebras. We begin with the definition of automorphic correction. More details on automorphic correction can be found in \cite{gritsenko1997siegel,gritsenko2002classification, kim2012rank, kim2013automorphic}.

\subsection{Definition}\label{automorphic}
A Kac-Moody algebra $\mathfrak g$ is called {\em Lorentzian} if its generalized Cartan matrix 
is given by a set of simple roots of a Lorentzian lattice $M$, namely, a lattice with a non-degenerate integral symmetric bilinear form 
$(\cdot, \cdot)$ of signature $(n,1)$ for some integer $n \ge 1$. 
A vector $\alpha\in M$ is a root if $(\alpha,\alpha) > 0$ and $(\alpha,\alpha)$ divides $2(\alpha,\beta)$ for all $\beta\in M$.
Let $\Pi$ be a set of (real) simple roots. Then the generalized Cartan matrix $A$ is given by
    \[ A= \begin{pmatrix} \frac {2(\alpha, \alpha')}{(\alpha, \alpha)} \end{pmatrix}_{\alpha, \alpha' \in \Pi}.
\]
The Weyl group $W$ is a subgroup of $O(M)$. 
Consider the cone \[ V(M)= \{ \beta \in M \otimes \mathbb R \, | \, (\beta,\beta) <0 \} ,\] which is a union of two half cones. One of these half cones is denoted by $V^+(M)$. The reflection hyperplanes of $W$ partition $V^+(M)$ into fundamental domains, and we choose one fundamental domain $\mathfrak D \subset V^+(M)$ so that the set $\Pi$ of (real) simple roots is orthogonal to the fundamental domain $\mathfrak D$. Then 
\[ \mathfrak D = \{ \beta \in \overline{V^+(M)} \, | \, (\beta, \alpha) \le 0  \text{ for all } \alpha \in \Pi \}.
\]
We have a Weyl vector $\rho \in M \otimes \mathbb Q$ satisfying $(\rho , \alpha) = - (\alpha, \alpha)/2$ for each $\alpha \in \Pi$.

Define the complexified cone $\Omega(V^+(M))= M \otimes \mathbb R + i V^+(M)$. Let $L=\begin{pmatrix} 0 & -m \\ -m & 0 \end{pmatrix} \oplus M$ be an extended lattice for some $m \in \mathbb N$. We consider the quadratic space $V=L \otimes \mathbb Q$ with the quadratic form induced from the bilinear form on $L$.
Let $V(\Bbb C)$ be the complexification of $V$ and $P(V(\Bbb C))=(V(\Bbb C)-\{0\})/\Bbb C^*$ be the corresponding projective space.
Let $\mathcal{K}^+$ be a connected component of
\begin{equation} \label{zz} \mathcal{K}=\{ [Z]\in P(V(\Bbb C)) : (Z,Z)=0,\, (Z,\bar Z)<0\},
\end{equation}
and let $O^+_V(\mathbb R)$ be the subgroup of elements in $O_V(\mathbb R)$ which preserve the components of $\mathcal K$.

For $Z\in V(\Bbb C)$, write $Z=X+iY$ with $X,Y\in V(\Bbb R)$.
Let $\Gamma \subseteq O^+_L:=O_L\cap O_V^+(\Bbb R) $ be a subgroup of finite index. Then $\Gamma$ acts on $\mathcal{K}$ discontinuously.
Set
$$\widetilde{\mathcal{K}}^+=\{ Z\in V(\Bbb C)-\{0\} : [Z]\in \mathcal{K}^+\}.
$$

Let $k\in  \frac 1 2 \mathbb Z$, and $\chi$ be a multiplier system of $\Gamma$. Then a meromorphic function $\Phi: \tilde{\mathcal{K}}^+\longrightarrow \Bbb C$ is called a {\em meromorphic modular form} of weight $k$ and
multiplier system $\chi$ for the group $\Gamma$, if
\begin{enumerate}
\item $\Phi$ is homogeneous of degree $-k$, i.e., $\Phi(cZ)=c^{-k}\Phi(Z)$ for all $c\in\Bbb C-\{0\}$,
\item $\Phi$ is invariant under $\Gamma$, i.e., $\Phi(\gamma Z)=\chi(\gamma)\Phi(Z)$ for all $\gamma\in \Gamma$.
\end{enumerate}

Define a map $\Omega(V^+(M)) \rightarrow \mathcal K$ by $z \mapsto \left [ \frac {(z, z)} {2m} \, e_1 +  e_2 + z \right ]$, where $\{ e_1, e_2 \}$ is the basis for $\begin{pmatrix} 0 & -m \\ -m & 0 \end{pmatrix}$. Then the space $\mathcal K^+$ is canonically identified with $\Omega(V^+(M))$.

Consider a meromorphic automorphic form
$\Phi(z)$ on $\Omega(V^+(M))$ with respect to a subgroup $\Gamma \subseteq O^+_L$ of finite index. The function  $\Phi(z)$ is called an 
{\em automorphic correction} of the Lorentzian Kac-Moody algebra $\mathfrak g$ if it has a Fourier expansion of the form
\begin{equation} \label{auto-corr}
\Phi(z) = \sum_{w \in W} \det (w) \left ( e \left ( -  (w (\rho), z) \right )- \sum_{a \in M \cap
     \mathfrak D,\, a\ne 0} m(a) \, e(- (w(\rho+a), z)) \right ),
\end{equation} 
where  $e(x)=e^{2 \pi i x}$ and $m(a) \in \mathbb Z$ for
     all $a \in M \cap \mathfrak D$.

We note that $O_V^+(\Bbb R)$ is the orthogonal group $O(n+1,2)$, and when $n=2$,  the automorphic forms on $O(3,2)$ are Siegel modular forms since $SO(3,2)$ is isogeneous to $Sp_4$. When $n=1$, which is our case in this paper, the automorphic forms on $O(2,2)$ are Hilbert modular forms  since $SO(2,2)$ is isogeneous to $SL_2\times SL_2$.
We also note that the denominator of $\mathfrak g$ is $\sum_{w \in W} \det (w) e \left ( -  (w (\rho), z) \right )$, which is not an automorphic form on $\Omega(V^+(M))$ in general, and one can see from \eqref{auto-corr} that the denominator of $\mathfrak g$ is corrected to be an automorphic form $\Phi(z)$.

An automorphic correction $\Phi(z)$ defines a generalized Kac-Moody superalgebra $\mathcal G$ as in \cite{gritsenko2002classification} so that the denominator of $\mathcal G$ is $\Phi(z)$. In particular, the function $\Phi(z)$ determines the set of imaginary simple roots of $\mathcal G$ in the following way:
First, assume that  $a \in M\cap \mathfrak D$ and $(a,a)<0$. If $m(a)>0$ then $a$ is an even imaginary simple root with multiplicity $m(a)$, and if $m(a)<0$ then $a$ is an odd imaginary simple root with multiplicity $-m(a)$. Next, assume that $a_0\in M\cap \mathfrak D$ is primitive and $(a_0,a_0)=0$. Then we define $\mu(na_0) \in \mathbb Z$, $n \in \mathbb N$ by 
\[  1 - \sum_{k=1}^\infty m(ka_0) t^k = \prod_{n=1}^\infty (1-t^n)^{\mu(na_0)} , \] where $t$ is a formal variable.
If $\mu(na_0) >0$ then $na_0$ is an even imaginary simple root with multiplicity $\mu(na_0)$; if $\mu(na_0)<0$ then $na_0$ is an odd imaginary simple root with multiplicity $-\mu(na_0)$. 

The generalized Kac-Moody superalgebra $\mathcal G$ will be also called an {\em automorphic correction} of $\mathfrak g$.
Using the denominator identity for $\mathcal G$, the automorphic form $\Phi(z)$ can be written as the infinite product
 \[ \Phi(z)= e(-(\rho, z)) \prod_{\alpha \in \Delta(\mathcal G)^+} (1 - e(-(\alpha, z)))^{\mathrm{mult}(\mathcal G, \alpha)} ,\] where $\Delta(\mathcal G)^+$ is the set of positive roots of $\mathcal G$ and $\mathrm{mult}(\mathcal G, \alpha)$ is the root multiplicity of 
$\alpha$ in $\mathcal G$.

\subsection{Rank $2$ hyperbolic Kac-Moody algebras}
Let $A=\begin{pmatrix} 2 & -a \\ -a & 2 \end{pmatrix}$ be a generalized Cartan matrix with $a \ge 3$, and  $\mathcal H(a)$ be the hyperbolic Kac-Moody algebra associated with the matrix $A$. We write $\mathfrak g = \mathcal H(a)$ if there is no need to specify $a$. Let $\{ h_1, h_2 \}$ be the set of simple coroots in the Cartan subalgebra $\mathfrak h = \mathbb C h_1 + \mathbb C h_2 \subset \mathfrak g$. Let $\{ \alpha_1 , \alpha_2 \} \subset \mathfrak h^*$ be the set of simple roots, and $Q=\mathbb Z \alpha_1 + \mathbb Z \alpha_2$ be the root lattice, and define $\mathfrak h^*_{\mathbb Q} =\mathbb Q \alpha_1 + \mathbb Q \alpha_2$ and $\mathfrak h^*_{\mathbb R} =\mathbb R \alpha_1 + \mathbb R \alpha_2$.  The set of roots of $\mathfrak g$ will be denoted by $\Delta$, and the set of positive (resp. negative) roots by $\Delta^+$ (resp. by $\Delta^-$), and the set of real (resp. imaginary) roots by $\Delta_{\mathrm{re}}$ (resp. by $\Delta_{\mathrm{im}}$). We will use the notation $\Delta^+_{\mathrm{re}}$ to denote the set of positive real roots. Similarly, we use $\Delta^+_{\mathrm{im}}$, $\Delta^-_{\mathrm{re}}$ and $\Delta^-_{\mathrm{im}}$.

Let $F=\mathbb Q(\sqrt {a^2-4})$, and let $N$ be the discriminant of $F$. We define $s \in \mathbb N$ by  $a^2-4=Ns^2$. We keep the notations in Section \ref{Notations} for the quadratic field $F$.
We set \[ \eta = \frac {a + \sqrt {a^2-4}} 2 = \frac {a+s\sqrt N} 2 .\] Then we have $\eta' = \eta^{-1}$ and $1+\eta^2=a \eta$.
The simple reflection corresponding to $\alpha_i$ in the root system of $\mathfrak g$ is denote by $r_i$ ($i=1, 2$), and the Weyl group by $W$. The eigenvalues of $r_1r_2$ as a linear transformation on $\mathfrak h^*$ are $\eta^2$ and $\eta^{-2}$. Let $\gamma^+$ be an eigenvector for $\eta^2$ and we set $\gamma^- =r_2 \gamma^+$. Then $\gamma^-$ is an eigenvector for $\eta'^{2}$. Specifically, we choose
\[ \gamma^+ = \frac {\alpha_1 + \eta' \alpha_2} s \quad \text{ and } \quad \gamma^- = \frac {\alpha_1 + \eta \alpha_2} s .\]
We define a symmetric bilinear form $(\cdot, \cdot)$ on $\mathfrak h^*$ to be given by the Cartan matrix $A$ with respect to $\{ \alpha_1, \alpha_2 \}$. Then we have $(\gamma^+, \gamma^+)=(\gamma^-, \gamma^-)=0$ and $(\gamma^+, \gamma^-)=-N$.

We will use the column vector notation for the elements in $\mathfrak h^*$ with respect to the basis $\{\gamma^+, \gamma^- \}$, i.e. we write $\binom  x y $ for $x \gamma^+ + y \gamma^-$. Then we have
\[ \alpha_1 = \frac 1 {\sqrt{N}} \begin{pmatrix} \eta \\ -  \eta' \end{pmatrix} \quad \text{ and } \quad \alpha_2= \frac 1 {\sqrt{N}}  \begin{pmatrix} - 1 \\  1 \end{pmatrix} . \] It follows that $\mathfrak h^*_{\mathbb Q}= \{ \binom x {\bar x} \, | \, x \in F\}$.
 A symmetric bilinear form $\langle \cdot , \cdot \rangle$ on $F$ is defined by $\langle x, y \rangle= - N \, \mathrm{tr}(xy')$. We define a map $\psi : \mathfrak h^*_{\mathbb Q} \rightarrow F$ by $\binom  x {\bar x}  \mapsto x$. Then the map $\psi$ is an isometry from $(\mathfrak h^*_{\mathbb Q}, (\cdot, \cdot))$ to $(F, \langle \cdot, \cdot \rangle)$.
 In particular, the root lattice $Q=\mathbb Z \alpha_1 + \mathbb Z \alpha_2$ is mapped onto a sublattice of $\mathcal O/\sqrt N$. 

Let $\omega_i$ $(i=1,2)$ be the fundamental weights of $\mathfrak g$. Then we have $\omega_1 = \frac 1 {4-a^2} ( 2 \alpha_1 + a \alpha_2)$ and $\omega_2 = \frac 1 {4-a^2}(a \alpha_1 + 2 \alpha_2)$. In the column vector notation,
\[ \omega_1 = \frac {-1} {sN} \begin{pmatrix} 1 \\ 1 \end{pmatrix} \quad \text{ and } \quad \omega_2= \frac {-1} {sN}  \begin{pmatrix} \eta \\ \bar \eta \end{pmatrix} . \] We define \[ \rho := - (\omega_1 + \omega_2) = \frac 1 {sN} \begin{pmatrix} 1+\eta \\ 1 + \bar \eta \end{pmatrix} . \]

 The simple reflections have the matrix representations
\[ r_1 = \begin{pmatrix} 0 &\eta^2 \\ \bar \eta^2 & 0 \end{pmatrix} \quad \text{ and } \quad r_2 = \begin{pmatrix} 0 & 1 \\ 1 & 0 \end{pmatrix} . \] The Weyl group $W$ also acts on $F$ by \[ r_1 x= \eta^2 \bar x \quad \text{ and } \quad r_2 x = \bar x \qquad \text{ for } x \in F,\] so that the isometry $\psi$ is $W$-equivariant.
Since $W=\{ (r_1r_2)^i, r_2(r_1r_2)^i \, | \, i \in \mathbb Z \}$, we calculate the set of positive real roots and obtain
\[ \Delta^+_{\mathrm{re}} = \left \{ \frac 1 {\sqrt{N}} \begin{pmatrix}  \eta^{j} \\ - \bar \eta^{j} \end{pmatrix} (j > 0),  \qquad
\frac 1 {\sqrt{N}} \begin{pmatrix}  - \bar \eta^{j} \\  \eta^{j} \end{pmatrix} (j \ge 0)  \right \} . \] We can also obtain a description of the set of positive imaginary roots. See  \cite{kang1995rank} for details.

\subsection{Hilbert modular forms} \label{subsec-H}
 
 We put $M=\psi^{-1}(\frak d^{-1}) \subset \frak h^*_{\mathbb Q} $.  Then $M$ is of signature $(1,1)$ and the Kac-Moody algebra $\mathfrak g= \mathcal H (a)$ is Lorentzian. We take the Weyl group $W$ for the reflection group of $M$, and choose the cone
\begin{equation} \label{eqn-V} V^+(M) = \{ x \gamma^+ + y \gamma^- \in \mathfrak h^*_{\mathbb R} \, |\,  x >0 , \ y>0 \}. \end{equation} We set $\Pi = \{ \alpha_1, \alpha_2 \}$ and obtain the Weyl chamber \[ \mathfrak D =\{ \beta \in V^+(M) \, | \, (\beta, \alpha_i) \le 0 , \ i=1,2 \} = \mathbb R_{\le 0}\, \omega_1+ \mathbb R_{\le 0}\, \omega_2 . \] The Weyl vector is given by  $\rho = -(\omega_1 +\omega_2)$. 

From our choice of $V^+(M)$ in \eqref{eqn-V}, we have the complexified cone \[ \Omega(V^+(M))= M \otimes \mathbb R + i V^+(M) = \left \{ \binom {z_1} {z_2} : Im(z_1)>0, \ Im(z_2)>0 \right \} \subset \mathfrak h^*\] with respect to the basis $\{ \gamma^+ , \gamma^- \}$. Then $\Omega(V^+(M))$  is naturally identified with $\mathbb H^2$. We choose the extended lattice $L=\begin{pmatrix} 0 & -N \\ -N & 0 \end{pmatrix} \oplus M$. Now it can be shown that an automorphic form on $\Omega(V^+(M))$ is a Hilbert modular form through the identification $\Omega(V^+(M) \cong \mathbb H^2$. See \cite{bruinier2008hilbert, kim2012rank} for details. 
As we identify $\mathbb H^2$ with $\Omega (V^+(M)) \subset \mathfrak h^*$,  the Weyl group $W$ acts on $\mathbb H^2$; in particular, we have \[ r_1(z_1, z_2) = (\eta^2 z_2,  \eta'^2 z_1) \quad \text{ and } \quad r_2(z_1, z_2) =(z_2, z_1) . \] We also define a paring on $F \times \mathbb H^2$ by
\begin{equation}\label{pairing}
(\nu, z)= -N\,(\nu z_2+ \nu' z_1),
\end{equation} 
for $\nu \in F$ and $z=(z_1, z_2) \in \mathbb H^2.$

Our automorphic correction will be a Hilbert modular form with respect to the congruence subgroup $\mathbf \Gamma_{0}(N)$ defined by
\[ \mathbf \Gamma_{0}(N) = \left \{ \begin{pmatrix} a & b \\ c & d \end{pmatrix} \in SL_2(\mathcal O_F) : a, b, d \in \mathcal O_F, \ c \in (N) \right \} \subset O^+_L ,\] where $(N) \subset \mathcal O_F$ is the principal ideal generated by $N$. 
We will need the following lemma, whose proof is essentially the same as that of Lemma 5.13 of \cite{kim2012rank}.
\begin{Lem} \label{lem-Nz}
Let $g(z)$ be a Hilbert modular form with respect to $SL_2(\mathcal O_F)$. Define $f(z)=\overline g(Nz)$, where $\overline g(z_1, z_2) = g(z_2, z_1)$. Then the function $f(z)$ is a Hilbert modular form with respect to the congruence subgroup $\mathbf \Gamma_0(N)$.
\end{Lem}

\subsection{Construction of automorphic correction}
Assume that $1 < N \le 21$. Then the modular form $f_1$ exists by Lemma \ref{lem-own}. In particular, we consider the following three cases: 
(1) $N=12$, $a=4$; (2) $N=8$, $a=6$; (3) $N=21$, $a=5$.
The fundamental units $\varepsilon_0$ are: $2+\sqrt{3}$, $1+\sqrt{2}$, $\frac {5+\sqrt{21}}2$, respectively, and $\eta=\varepsilon_0$ for $N=12, 21$, and $\eta=\varepsilon_0^2$ for $N=8$.

We choose the Weyl chamber $\mathcal W$ attached to $f_1$ which contains the point $(1, \varepsilon_0)$. Then the Weyl vector is given by the formula \eqref{Weyl-vector}. In the case $N(\varepsilon_0)=-1$, the point $(\varepsilon_0^{-1}, \varepsilon_0)$ lies in the same Weyl chamber $\mathcal W$. (See Example 3.11 in \cite{bruinier2008hilbert}.) 

Recall that we have the isometry $\psi: \mathfrak h^*_{\mathbb Q} \rightarrow F$ by $\binom \nu {\nu'} \mapsto  \nu$. One can check that 
\begin{equation} \label{eqn-rw}   \psi (\rho)= \frac 1 {sN} (1 +\eta) = \rho_{f_1, \mathcal W} , \end{equation} and we write $\rho = \rho_{f_1, \mathcal W}$ if there is no peril of confusion.

First, assume that $\nu \gg 0$ and $(\nu, \mathcal W) >0$ for $\nu \in \mathfrak d^{-1}$. Then  $N(\nu)>0$ and $\langle \nu, \nu \rangle = -N \text{tr}(\nu\nu') <0$. Thus $\nu$ corresponds to an imaginary root of $\mathcal H(a)$ by Proposition 5.10 \cite{kac1990infinite}. We can also check $\langle \rho, \nu \rangle <0$ using $\nu+\varepsilon_0 \nu' >0$ or $\varepsilon_0^{-1}\nu+ \varepsilon_0 \nu'>0$ if $N(\varepsilon_0)=-1$. Thus $\nu$ is positive, i.e. $\nu \in \psi(\Delta^+_{\rm{im}})$.

Next, assume that $\nu \not\gg 0$ and $(\nu, \mathcal W) >0$ for $\nu \in \mathfrak d^{-1}$. Then $N(\nu)<0$ and $a(N\nu \nu') \neq 0$ only for $\nu$ with $N(\nu)=\nu \nu'=-1/N$, in which case $s(N\nu \nu')a(N\nu \nu')=1$. Further, $\langle \nu, \nu \rangle = 2$, and $\nu$ corresponds to a positive real root of $\mathcal H(a)$, i.e. $\nu \in \psi(\Delta^+_{\rm{re}})$. 

Now, from the above observations, the Borcherds product \eqref{eqn-BB} can be written:
\begin{align*}  \Psi(z_1,z_2) &=e(\rho z_1+\rho'z_2)\prod_{\substack{ \nu\in\mathfrak d^{-1} \\ (\nu,\mathcal W)>0} } \left(1-e(\nu z_1+\nu' z_2) \right)^{s(N\nu\nu')a(N\nu\nu')} \\ &= e(\rho z_1+\rho'z_2)\prod_{\substack{ \nu\in\mathfrak d^{-1} \\ \nu \gg 0 } }\left(1-e(\nu z_1+\nu' z_2) \right)^{s(N\nu\nu')a(N\nu\nu')} \\ & \phantom{LLLLLLLLLLLLLL} \times \prod_{\substack{ \nu\in\mathfrak d^{-1}, \nu + \varepsilon_0 \nu' >0 \\ N(\nu)=-1/N}}\left(1-e(\nu z_1+\nu' z_2) \right) \\
&= e(\rho z_1+\rho'z_2)\prod_{\nu\in \psi(\Delta^+_{\rm{im}}) }\left(1-e(\nu z_1+\nu' z_2) \right)^{s(N\nu\nu')a(N\nu\nu')} 
\prod_{\nu\in \psi(\Delta^+_{\rm{re}}) }\left(1-e(\nu z_1+\nu' z_2) \right) \\
\end{align*}
Define $\Phi(z) =\Phi(z_1,z_2) =\Psi(Nz_2,Nz_1)$. Then by Lemma \ref{lem-Nz}, $\Phi$ is a Hilbert modular form with respect to 
$\mathbf\Gamma_0(N)$. By (\ref{pairing}), $\Phi$ can be written as
$$\Phi(z)=e(-(\rho,z))\prod_{\nu\in \psi(\Delta^+_{\rm{im}}) }
(1-e(-(\nu,z))^{s(N\nu \nu')a(N\nu \nu')} \prod_{\nu\in \psi(\Delta^+_{\mathrm{re}}) } \left(1-e(-(\nu,z))\right).
$$

As in \cite{kim2012rank}, we can prove
\[ \Phi(wz) =\det(w) \Phi (z) \quad \text{ for } w \in W. \]
This in turn implies that $\Phi(z)$ can be written as
$$\Phi(z)=\sum_{w\in W} \left(e(-(w(\rho),z))-\sum_{\nu\in M\cap \mathfrak D, \nu\ne 0} m(\nu) e(-(w(\rho+\nu),z))\right).
$$
This is exactly the form for the automorphic correction in (\ref{auto-corr}), and hence it provides an automorphic correction for $\mathcal H(a)$. So we have obtained:

\begin{Thm}\label{auto-corr}
Let $\mathcal H(a)$ be the rank $2$ symmetric hyperbolic Kac-Moody algebra, and let $N$ be the discriminant of the quadratic field $\mathbb Q(\sqrt{a^2-4})$. Assume that $1<N\le 21$. Then the Hilbert modular form $\Phi$
provides an automorphic correction for the hyperbolic Kac-Moody algebra $\mathcal H(a)$. In particular, there exists a generalized Kac-Moody superalgebra $\widetilde{\mathcal H}$ whose denominator function is the Hilbert modular form $\Phi$. 
\end{Thm}

\begin{Rmk} \label{rmk-rmk}
With the results in \cite{kim2012rank}, we have established automorphic correction for the hyperbolic Kac-Moody algebra $\mathcal H(a)$ corresponding to each of the fundamental discriminants $1 <N \le 21$. We summarize it in a table:
\[\begin{array}{c||c|c|c|c|c|c}
N& 5&8&12&13&17&21\\
\hline
a &3& 6&4&11&66&5\\
\end{array}
\] 
If $N > 21$, the modular form $f_1$ does not exist by Lemma \ref{lem-own}, 
and we need a new idea to construct automorphic correction.
\end{Rmk}

\section{Asymptotics and Positivity of Fourier Coefficients} \label{sec-Asym}

In this section, we obtain asymptotics of Fourier coefficients of the modular forms $f_m$ defined in Section 2. Note that the Fourier coefficients of $f_1$ are root multiplicities of the generalized Kac-Moody superalgebra $\widetilde{\mathcal H}$ with some modification.

We call a cusp $s\sim\frac{1}{c}$  with $c\mid N$ \emph{irregular} if $1<(c,N_2)<N_2$. Here $N_2$ is the $2$-part of $N$. We first prove a lemma on the holomorphy of $f_N$ at irregular cusps.
\begin{Lem}\label{Holo-Irr}
If $f_N$ exists, it is holomorphic at irregular cusps.
\end{Lem}
\begin{proof}
To save space, we employ the notations in \cite{scheithauer2009weil}. Let $F=\sum_\gamma F_\gamma e_\gamma$ be the vector valued modular form that corresponds to $f_N$ under the isomorphism in Theorem 4.16 of \cite{zhang2013vector}. Let $s\sim \frac{1}{c}$ be irregular. To obtain the holomorphy at $s$, we only need to consider that of $F_0|W(N)M$ at $\infty$ for some $M\in\text{SL}_2(\mathbb Z)$ with $M\infty\sim s$, where $W(N)= {\scriptsize \begin{pmatrix}0&-1\\N&0\end{pmatrix} }$.  Because $W(N)M\infty\sim \frac{1}{N/c}$, there exists $M'\in\text{SL}_2(\mathbb Z)$ and a parabolic matrix $P\in \text{GL}_2(\mathbb Q)$, such that $F_0|W(N)M=F_0|M'P$. Since $P$ does not affect the holomorphy, we only need to consider $F_0|M'$. It is clear that the left lower entry $c'$ of $M'$ satisfies $(c',N)=\frac{N}{c}$.

Now by Theorem 4.7 in \cite{scheithauer2009weil}, $F_0|M'$ belongs to the space spanned by $F_\beta$ with $\beta\in D^{c'*}$. (See page 6 in \cite{scheithauer2009weil} for explanation of notations.) From the particular form of the $2$-adic component of our discriminant form and the fact that $1<(c',N_2)<N_2$, we have $x_{c'}\neq 0$ and $q(\beta)\neq 0$ for any $\beta\in D^{c'*}$. Now from the explicit isomorphism in Theorem 4.16 of \cite{zhang2013vector}, we see that $F_\beta$ is holomorphic for any such $\beta$, since $F_\beta$ only collects the Fourier coefficients $a(n)$ of $f_N$ such that $n\equiv Nq(\beta)\imod N$. Hence the holomorphy of $f_N$ at irregular cusps follows.
\end{proof}

We recall the result of J. Lehner \cite{lehner1964discontinuous} on Fourier coefficients of modular forms using the method of Hardy-Ramanujan-Rademacher.
We refer to \cite{lehner1964discontinuous,lehner1964automorphic} for unexplained notations:
Let $f(\tau)$ be a weakly homomorphic modular form of weight $0$ with respect to $\Gamma$. Let $s_0=\infty, s_1,...,s_{l-1}$ be the set of inequivalent cusps of $\Gamma$,
and
$$A_0=\begin{pmatrix} 1&0\\0&1\end{pmatrix},\quad A_j=\begin{pmatrix} 0&-1\\1&-s_j\end{pmatrix},\, j>0.
$$
Let $M^*=A_j M=\begin{pmatrix} a&b\\c&d\end{pmatrix}$ for $M\in \Gamma$, and let
\begin{eqnarray*} C_{j0} &=& \{ c\, |\, \begin{pmatrix} \cdot &\cdot\\ c&\cdot\end{pmatrix}\in A_j\Gamma \}, \\
                  D_c    &=& \{ d\, |\, \begin{pmatrix} \cdot &\cdot\\ c&d\end{pmatrix}\in A_j\Gamma,\, 0<d\leq c \}.
\end{eqnarray*}
For $j=1,...,l-1$, let
$$e \left ( - \tfrac {\kappa_j} {\lambda_j} {A_j \tau} \right ) f(\tau)=\sum_{n=-\mu_j}^\infty a(n)^{(j)} q_j^n,\quad q_j=e \left ( \tfrac {A_j \tau}{\lambda_j} \right ),
$$
where $\kappa_j, \lambda_j$ are defined as in \cite{lehner1964automorphic}, page 398.
By replacing $A_j \tau$ by $\tau$, this can be written as
$$f(s_j-\tfrac 1\tau)= q^{\frac {\kappa_j}{\lambda_j}} \sum_{n=-\mu_j}^\infty a(n)^{(j)} q^{\frac n{\lambda_j}}.
$$
For $j=0$, we have the usual Fourier expansion: (We assume that $\lambda_0=1, \kappa_0=0$ for $\Gamma$.)
$$f(z)=\sum_{n=-\mu_0}^\infty a(n) q^n.
$$

\begin{Thm}[\cite{lehner1964discontinuous}] For $n>0$,
\begin{equation} \label{fourier} a(n)=2\pi  \sum_{j=0}^{l-1} \sum_{\nu=1}^{\mu_j} a(-\nu)^{(j)} \sum_{c\in C_{j0}} c^{-1}A(c,n,\nu_j)M(c,n,\nu_j,0),
\end{equation}
where $\nu_j=\frac {\nu-\kappa_j}{\lambda_j}$, and
\begin{eqnarray*}
A(c,n,\nu_j) &=& \sum_{d\in D_c} v^{-1}(M) e \left ( \frac {nd-\nu_ja}c \right ), \quad M=A_j^{-1} M^*, \\
M(c,n,\nu_j,0) &=& \left (\frac {\nu_j}n \right )^{\frac 12} I_1 \left(\frac {4\pi\sqrt{n\nu_j}}c \right).
\end{eqnarray*}
\end{Thm}

Consider $\Gamma=\Gamma_0(12)$. The cusps are $s_0=\infty$, $s_1=\frac 12$, $s_2=\frac 14$, $s_3=\frac 13$, $s_4=\frac 16$, $s_5=0$, among which $s_1$ and $s_4$ are irregular.
Then $\kappa_{j}=0$, and 
$$\lambda_{0}=1,\quad \lambda_{1}=12,\quad \lambda_{2}=48,\quad
\lambda_{3}=36,\quad \lambda_{4}=36,\quad \lambda_{5}=12.
$$
We first consider $f_1$. Since $f_1$ is holomorphic at all other cusps except at $\infty$ (Corollary 4.14 of \cite{zhang2013vector}),
$$a(n)=2\pi \frac {A(12,n,1)}{12} M(12,n,1,0)+ \text{error term},
$$
where $M(12,n,1,0)=\frac 1{\sqrt{n}} I_1\left(\frac {4\pi\sqrt{n}}{12}\right)$ and
\begin{eqnarray*}
A(12,n,1) &=& \sum_{d\, \text{mod 12}} \chi_3(d) e\left(\frac {nd-a}{12}\right)\\ &=&e\left(\frac {n+11}{12}\right)-e\left(\frac {5n-5}{12}\right)-e\left(\frac {7n-7}{12}\right)+e\left(\frac {11n+1}{12}\right) \\
 &=& 4\sin\frac {\pi (n-1)}{2} \sin\frac {\pi (n-1)}3.
\end{eqnarray*}
Next we consider $f_{12}$. Note the following Fourier expansions along various cusps: $f_{12}$ is holomorphic at $\frac 12$, $\frac 16$ by Lemma \ref{Holo-Irr}, and by the $\epsilon$-condition,
\begin{eqnarray*}
f_{12}\left(\frac 13-\frac 1\tau\right) &=& -\frac i2 q^{-\frac 1{12}}+O(1);\\
f_{12}\left(\frac 14-\frac 1\tau\right) &=& -\frac {i\sqrt{3}}4 q^{-\frac 1{12}}+O(1); \\
f_{12}\left(-\frac 1\tau\right) &=& \frac {\sqrt{3}}2 q^{-\frac 1{12}}+O(1).
\end{eqnarray*}

We can see them as follows: Let $s=\frac 14$. By explicit computation, we may choose
\[\gamma_4=\begin{pmatrix}81&80\\ -1376&-1359\end{pmatrix}=\begin{pmatrix}409&-82\\ -6948&1393\end{pmatrix}\begin{pmatrix}1&2\\ 4&9\end{pmatrix}\in\Gamma_0(12)\begin{pmatrix}1&2\\ 4&9\end{pmatrix}.\]
Since $\chi_D(1393)=1$, we have
\[f_{12}|\gamma_4=f_{12}\left|\begin{pmatrix}1&2\\ 4&9\end{pmatrix}\right. . \]
Therefore
\[f_{12}\left(\frac 14-\frac 1\tau\right)=f\left|\begin{pmatrix} 1&-4\\4&0\end{pmatrix}\right.=f_{12}\left|\begin{pmatrix}1&2\\ 4&9\end{pmatrix}\begin{pmatrix}1&2\\ 4&9\end{pmatrix}^{-1}\begin{pmatrix}1&-4\\ 4&0\end{pmatrix}\right.=f_{12}\left|\gamma_4\begin{pmatrix}1&-36\\ 0&16\end{pmatrix}\right..\]
Since $f_{12}|\gamma_4=\frac{-\sqrt{3}i}{4}q^{-\frac{4}{3}}+O(1)$,
\[f_{12}\left(\frac 14-\frac 1\tau\right)=\frac{-\sqrt{3}i}{4}e(-3)q^{-\frac{1}{12}}+O(1)=\frac{-\sqrt{3}i}{4}q^{-\frac{1}{12}}+O(1).\]
Similarly, for the cusp $s=\frac 13$, let
\[\gamma_3=\begin{pmatrix}64&63\\ 1089&1072\end{pmatrix}=\begin{pmatrix}67&-1\\ 1140&-17\end{pmatrix}\begin{pmatrix}1&1\\ 3&4\end{pmatrix}.\]
Since $\chi(-17)=-1$, we have
\[f_{12}\left(\frac 13-\frac 1\tau\right)=f_{12}\left|\begin{pmatrix}1&-3\\ 3&0\end{pmatrix}\right.=f_{12}\left|\begin{pmatrix}1&1\\ 3&4\end{pmatrix}\begin{pmatrix}1&-12\\ 0&9\end{pmatrix}=-(f_{12}|\gamma_3)\left(\frac{\tau-12}{9}\right)\right..\]
Since $f_{12}|\gamma_3=\frac{i}{2}q^{-\frac{3}{4}}+O(1)$, we have
\[f_{12}\left(\frac 13-\frac 1\tau\right)=-\frac{i}{2}q^{-\frac{1}{12}}+O(1).\]
The case when $s=0$ is simpler, since we can just take $\gamma_6=S$. Therefore,
\[f_{12}\left(-\frac{1}{\tau}\right)=f_{12}|S=f_{12}|\gamma_6=\frac{\sqrt 3}{2}q^{-\frac{1}{12}}+O(1).\]
So the main term of $a(n)$ is
\begin{eqnarray*}
 &{}& 2\pi \Big \{ \frac 1{48} \left( e\left(\frac n{12}\right)+e\left(\frac {11n}{12}\right)-e\left(\frac {5n}{12}\right)-e\left(\frac {7n}{12}\right)\right) M(12,n,12,0)
 + \frac {\sqrt{3}}2 M(1,n,\frac 1{12},0) \\
 &+&\frac i2 \left(e\left(\frac n3\right)-e\left(\frac {2n}3\right)\right) M(1,n,\frac 1{12},0)+
\frac {i\sqrt{3}}4 \left(e\left(\frac n4\right)-e\left(\frac {3n}4\right)\right) M(1,n,\frac 1{12},0) \Big \}.
\end{eqnarray*}
and we have
$$a(n)=\frac {\pi}{\sqrt{n}} \left(\frac 1{\sqrt{3}}\left(\sin\frac {\pi n}2 \sin \frac {\pi n}3-\sin\frac {2\pi n}3\right)+
\frac 12\left(1-\sin \frac {\pi n}2\right)\right) I_1\left(\frac {4\pi\sqrt{n}}{\sqrt{12}}\right)+ \text{error term}.
$$

For example, when $n=18$, the main term is 177366 and the actual value of $a(18)$ is 177246. When $n=15$, the main term is 106127 and the actual value is 106144.
When $n=17$, the main term vanishes, and it agrees with the numerical data. 

The error term can be computed as in \cite{kim2012rank}, and we can show that $a(n)$ is always non-negative. We summarize the above results as 
\begin{Prop} For $\Gamma_0(12)$, let $f_m=s(m)^{-1}q^{-m}+\sum_{n=0}^\infty a_m(n)q^n$. Then for $f_1$,

$$a_1(n)=\frac {2\pi}{3\sqrt{n}} \sin\frac {\pi (n-1)}{2} \sin\frac {\pi (n-1)}3 I_1\left(\frac {\pi\sqrt{n}}3\right)+O\left(\frac {8\pi^{\frac 32}\log\frac {\pi\sqrt{n}}3}{3 n^{\frac 14}} I_1\left(\frac {\pi\sqrt{n}}6\right)\right).
$$

For $f_{12}$, $a_{12}(n)\geq 0$ for all $n$, and
\begin{eqnarray*}
a_{12}(n) &=& \frac {\pi}{\sqrt{n}} \left(\frac 1{\sqrt{3}}\left(\sin\frac {\pi n}2 \sin \frac {\pi n}3-\sin\frac {2\pi n}3\right)+
\frac 12\left(1-\sin \frac {\pi n}2\right)\right) I_1\left(\frac {4\pi\sqrt{n}}{\sqrt{12}}\right) \\
&+& O\left(16\sqrt{\pi} (12 n)^{\frac 14}\log \frac {4\pi\sqrt{n}}{\sqrt{12}} I_1\left(\frac {2\pi\sqrt{n}}{\sqrt{12}}\right)\right).
\end{eqnarray*}
\end{Prop}

Consider $\Gamma=\Gamma_0(8)$. The cusps are $s_0=\infty$, $s_1=\frac 12$, $s_2=\frac 14$, and $s_3=0$, of which $s_1$ and $s_2$ are irregular.
Since $f_1$ is holomorphic at all other cusps except at $\infty$ (Corollary 3.6 of \cite{zhang2013zagier}),
$$a_1(n)=\frac {2\pi}8 A(8,n,1)M(8,n,1,0)+\text{error term},
$$
where $M(8,n,1,0)=\frac 1{\sqrt{n}} I_1\left(\frac {4\pi\sqrt{n}}{8}\right)$, and 
$$A(8,n,1)=\sum_{d\, \text{mod 8}} \chi_{8}(d) e\left(\frac {nd-a}{8}\right)=4\sin\frac {\pi (n-1)}2 \sin\frac {\pi (n-1)}4.
$$
Consider $f_8$: By Lemma \ref{Holo-Irr}, $f_8$ is holomorphic at $\frac 12$ and $\frac 14$, and
$$f_8\left(-\frac 1\tau\right)=\sqrt{2} q^{-\frac 18}+O(1).
$$
Hence the main term of $a_8(n)$ is
\begin{eqnarray*}
 &{}& 2\pi \Big \{ \frac 1{16} \left( e\left(\frac n{8}\right)+e\left(\frac {7n}{8}\right)-e\left(\frac {3n}{8}\right)-e\left(\frac {5n}{8}\right)\right) M(8,n,8,0)
 + \sqrt{2} M(1,n,\frac 1{8},0) \Big \} \\
&{}& =\frac {\pi}{\sqrt{n}} I_1(\pi\sqrt{2n}) \left(1+\sqrt{2}\sin\frac {\pi n}2 \sin\frac {\pi n}4\right).
\end{eqnarray*}

When $n=17$, this is $1.2557\times 10^7$ and the actual value of $a_8(17)$ is 12556992. When $n=18$, this is $1.02365\times 10^7$ and the actual value is 10235352. We can show easily that $a_8(n)$ is always non-negative. We summarize our results as 
\begin{Prop} For $\Gamma_0(8)$, let $f_m=s(m)^{-1}q^{-m}+\sum_{n=0}^\infty a_m(n)q^n$. Then for $f_1$,
$$a_1(n)=\frac {\pi}{\sqrt{n}} \sin\frac {\pi (n-1)}2 \sin\frac {\pi (n-1)}4 I_1\left(\frac {\pi\sqrt{n}}{2}\right)
+O\left( \frac {2\sqrt{\pi}\log\frac {\pi\sqrt{n}}2}{n^{\frac 14}} I_1\left(\frac {\pi\sqrt{n}}4\right)\right).
$$
For $f_{8}$, $a_{8}(n)\geq 0$ for all $n$, and
\begin{equation*}
a_{8}(n)=\frac {\pi}{\sqrt{n}} I_1(\pi\sqrt{2n}) \left(1+\sqrt{2}\sin\frac {\pi n}2 \sin\frac {\pi n}4\right)+
O\left(\frac {2^{\frac 14} \sqrt{\pi}\log \pi\sqrt{2n}}{n^{\frac 14}} I_1\left(\frac {\pi\sqrt{n}}{\sqrt{2}}\right)\right).
\end{equation*}
\end{Prop}

Consider $\Gamma=\Gamma_0(21)$. The cusps are $s_0=\infty$, $s_1=\frac 13$, $s_2=\frac 17$, and $s_3=0$.
Then since $f_1$ is holomorphic at all other cusps except at $\infty$ (Corollary 3.6 in \cite{zhang2013zagier}),
$$a_1(n)=\frac {2\pi}{21} A(21,n,1)M(21,n,1,0)+\text{error term},
$$
where $M(21,n,1,0)=\frac 1{\sqrt{n}} I_1\left(\frac {4\pi\sqrt{n}}{21}\right)$, and 
$$A(21,n,1)=\sum_{d\, \text{mod 21}} \chi_{21}(d) e\left(\frac {nd-a}{21}\right)=\sqrt{21} \left(\frac n{21}\right) \sum_{v^2\equiv -n \,\text{mod 21}} e\left(\frac {2v}{21}\right).
$$

We summarize our result as 
\begin{Prop} For $\Gamma_0(21)$, let $f_m=s(m)^{-1}q^{-m}+\sum_{n=0}^\infty a_m(n)q^n$. Then for $f_1$,

$$a_1(n)=\frac {2\pi}{\sqrt{21 n}} I_1\left(\frac {4\pi\sqrt{n}}{\sqrt{21}}\right) \left(\frac n{21}\right) \sum_{v^2\equiv -n \,\text{mod 21}} e\left(\frac {2v}{21}\right) 
+O\left(\frac {4\pi^{\frac 32}\log\frac {4\pi\sqrt{n}}{21}}{21 n^{\frac 14}} I_1\left(\frac {2\pi\sqrt{n}}{21}\right)\right).
$$
\end{Prop}

Similar computations can be applied to other cases. 

\vskip 0.5 cm

\bibliographystyle{amsplain}
\bibliography{paper}

\end{document}